\documentclass[reqno,10pt]{amsart}

\usepackage{amsmath}
\usepackage{amsfonts}
\usepackage{amsthm}
\usepackage{amssymb}
\usepackage{graphicx}
\usepackage{color}
\usepackage{float}
\usepackage{url}

\numberwithin{equation}{section}
\numberwithin{figure}{section}

\DeclareMathOperator{\curl}{curl}
\DeclareMathOperator{\Div}{div}
\DeclareMathOperator{\Spec}{Spec}
\newtheorem{theorem}{Theorem}[section]
\newtheorem{lemma}[theorem]{Lemma}
\newtheorem{proposition}[theorem]{Proposition}

\newtheorem{remark}[theorem]{Remark}
\newtheorem{corollary}[theorem]{Corollary}

\let\epsilon\varepsilon

\title{Superconductivity between $H_{C_2}$ and $H_{C_3}$}

\author{S. Fournais}
\author{B. Helffer}
\author{M. Persson}

\address[S. Fournais and M. Persson]{Department of Mathematical Sciences\\
University of Aarhus \\Ny Munkegade 118\\ 8000 Aarhus C \\ Denmark
}
\email{fournais@imf.au.dk, mickep@imf.au.dk}
\address[B. Helffer]{CNRS and Laboratoire de
Math\'{e}matiques UMR CNRS 8628\\ Universit\'{e} Paris-Sud - B\^{a}t 425\\
F-91405 Orsay Cedex\\ France.}
\email{bernard.helffer@math.u-psud.fr}

\date{\today}

\begin{document}

\begin{abstract}
Superconductivity for Type II superconductors in external magnetic
fields of magnitude between the second and third critical fields is
known to be restricted to a narrow boundary region. The profile of the
superconduc\-ting order parameter in the Ginzburg-Landau model is
expected to be governed by an effective one-dimensional model. This is
known to be the case for external magnetic fields sufficiently close
to the third critical field. In this text we prove such a result on a
larger interval of validity.
\end{abstract} 

\maketitle

\section{Introduction}

\subsection{Background}
When studying superconductivity in the Ginzburg-Landau mo\-del in strong
magnetic fields, one encounters three critical values of the magnetic
field strength. The first critical field is where a vortex appears and
will not concern us in the present text. At the second critical field,
denoted $H_{C_2}$,
superconductivity becomes essentially restricted to the boundary and
is weak in the interior. At the third critical field, $H_{C_3}$,
superconductivity disappears altogether.
In this paper we will discuss superconductivity in the zone between
$H_{C_2}$ and $H_{C_3}$.

The Ginzburg-Landau model of superconductivity is the following
functional,
\begin{align}\label{eq-hc2-GL}
\mathcal E[\psi,\mathbf{A}]&=\int_\Omega |(\nabla-i\kappa
H\mathbf{A})\psi|^2-\kappa^2|\psi|^2+\frac{\kappa^2}2|\psi|^4+(\kappa
H)^2|{\rm curl}(\mathbf{A}-\mathbf{F})|^2\,dx\,.
\end{align}
Here $\psi \in W^{1,2}(\Omega)$ is a complex valued wave function, 
$\mathbf{A} \in W^{1,2}(\Omega,{\mathbb R}^2)$ a
vector potential, $\kappa$ the Ginzburg-Landau parameter (a material
parameter), and  $H$ is the
strength of the applied magnetic field. The potential
$\mathbf{F}:\Omega\to{\mathbb R}^2$ is the unique vector field satisfying,
\begin{equation}\label{eq-hc2-F}
\curl\mathbf{F}=1\,,\quad\Div\mathbf{F}=0\quad\text{in}~\Omega\,,
\qquad\qquad
N\cdot \mathbf{F}=0\quad\text{on}~\partial\Omega\,,
\end{equation}
where $N$ is the unit inward normal vector of $\partial\Omega$.

With this notation, the critical fields behave as follows for large
$\kappa$:
\begin{align}
  \label{eq:24}
  H_{C_2} \approx \kappa + o(\kappa),\qquad H_{C_3} \approx
  \frac{\kappa}{\Theta_0} + o(\kappa),
\end{align}
where $\Theta_0 \approx 0.59$ is a universal constant. The definition of 
$\Theta_0$ is recalled in~\eqref{eqTheta_0} below.

Therefore, when we study the Ginzburg-Landau functional for $H = b \kappa$, 
$1<b<\Theta_0^{-1}$, superconductivity should be a boundary
phenomenon. This was proved in a weak sense in~\cite{pan2}.

\begin{theorem}[\cite{pan2}]\label{thm:Pan}
For any $b\in \left]1, \Theta_0^{-1}\right[$, there exists a constant $E_b$, 
such that, for $H = \kappa b$, 
\begin{align}
  \label{eq:25}
  \inf_{(\psi,{\bf A}) \in W^{1,2}(\Omega) \times
    W^{1,2}(\Omega;{\mathbb R}^2)}
  {\mathcal E}_{\kappa, H} [\psi, {\bf A}] = - \sqrt{\kappa H} E_b
  |\partial \Omega| + o(\kappa),\qquad \text{ as } \kappa \rightarrow \infty.
\end{align}
\end{theorem}
Local energy results are also obtained in~\cite{pan2}. Theorem~\ref{thm:Pan} 
indicates that superconductivity is uniformly distributed along the
boundary. However, the constant $E_b$ is only defined as 
 a limit and  its calculation is not easy. A number
of conjectures related to the calculation of $E_b$ are given in~\cite{pan2}.
In~\cite{AlHe} (see also~\cite[Chapter~14]{fohebook}), the constant $E_b$ is
determined for $b$ in the vicinity of $\Theta_0^{-1}$. It turns out
that the determination of
the constant in 
this {\it non-linear} problem can be reduced to the positivity of a
{\it linear} operator. 
Define the space ${\mathcal B}^1({\mathbb R}^+)$ as
\begin{align}
  \label{eq:29}
  {\mathcal B}^1({\mathbb R}^+) = \{ \phi \in L^2({\mathbb R}^+) \,:\,
  \phi' \in  L^2({\mathbb R}^+) \text{ and } t \phi \in  L^2({\mathbb R}^+)\}.
\end{align}
Define, for $z \in {\mathbb R}$, $\lambda>0$,
\begin{align}
  \label{eq:44}
  {\mathcal F}_{z,\lambda}(\phi)&:= 
\int_0^{+\infty}|\phi'(t)|^2 + (t-z)^2 |\phi(t)|^2 +
\frac{\lambda}{2} |\phi(t)|^4 - \lambda |\phi(t)|^2\,dt\,,
\end{align}
and let $f_{z,\lambda}$ be a non-negative minimizer of this functional
(see Theorem~\ref{thm:Sammenkog} below for properties of
minimizers---in particular the fact that $f_{z,\lambda}$ exists and is unique).

For given $\lambda >0$,  minimize ${\mathcal F}_{z,\lambda}(f_{z,\lambda})$ over
$z$ and denote a minimum by $\zeta(\lambda)$---we will prove below that such a 
minimum exists when $\lambda \in \left]\Theta_0,1\right]$. 
By definition of $f_{\zeta(\lambda) ,\lambda}$,
\begin{align}\label{eq:26}
  {\mathcal F}_{z,\lambda}(\phi) 
\geq {\mathcal F}_{\zeta(\lambda),\lambda}(f_{\zeta(\lambda),\lambda}),
\end{align}
for all $(z,\phi) \in {\mathbb R} \times {\mathcal B}^1({\mathbb R^+})$. 

We also introduce a linear operator ${\mathfrak k}_{\lambda}$. 
Define, for $\nu \in {\mathbb R}$, $\lambda \in {\mathbb R^+}$, the operator 
${\mathfrak k}_{\lambda}= {\mathfrak k}_{\lambda}(\nu)$ to be the Neumann 
realization of
\begin{align}\label{eq:1b}
  {\mathfrak k}_{\lambda}(\nu) = -\frac{d^2}{dt^2} + (t - \nu)^2 
  + \lambda f_{\zeta(\lambda),\lambda}(t)^2,
\end{align}
on $L^2({\mathbb R}_{+})$. We denote by 
$\{\lambda_j(\nu)\}_{j=1}^{\infty}$ the spectrum of 
${\mathfrak k}_{\lambda}(\nu)$. Also 
$\{ v_j(t;\nu)\}_{j=1}^{\infty}$ will be the associated real, 
normalized eigenfunctions.

\begin{remark}
Notice the following complication: Since we do not know that $\zeta(\lambda)$ is
unique, the operator ${\mathfrak k}_{\lambda}(\nu)$ is really a family of 
operators,
\[
{\mathfrak k}^{(j)}_{\lambda}(\nu) =  
-\frac{d^2}{dt^2} + (t - \nu)^2 + \lambda f_{\zeta_j(\lambda),\lambda}(t)^2,
\]
one for every minimum $\zeta_j(\lambda)$.
\end{remark}

It follows from~\cite{AlHe,fohebook} that

\begin{theorem}\label{eq:SpecThm}
Let $\lambda \in \left]\Theta_0,1\right[$. 
Suppose that there exists a minimum $\zeta(\lambda)$ such that for the 
corresponding choice of the operator ${\mathfrak k}_{\lambda}(\nu)$ we have
\begin{align}
  \label{eq:27}
\lambda \leq \inf_{\nu\in\mathbb{R}} \lambda_1(\nu)\,.
\end{align}
Then
\begin{align}
  \label{eq:28}
  E_{\lambda^{-1}} = \frac{\lambda}{2} \|
  f_{\zeta(\lambda),\lambda}\|_{L^4({\mathbb R}^+)}^4.
\end{align}
\end{theorem}

It is also proved in~\cite{AlHe,fohebook} (see Proposition~14.2.13 
in~\cite{fohebook}) that there exists $\epsilon >0$ such 
that~\eqref{eq:27} is satisfied for 
$\lambda \in \left]\Theta_0, \Theta_0+\epsilon\right[$.
The objective of the present paper is to give explicit bounds on the magnitude 
of $\epsilon$.

\begin{remark}
A minimizer $f_{z,\lambda}$ of the functional ${\mathcal F}_{z,\lambda}$ will
be a solution to the Euler-Lagrange equations for the minimization 
problem~\eqref{eq:44}
\begin{align}\label{eq:9b}
-u'' + (t-z)^2 u + \lambda |u|^2 u = \lambda u, \qquad u'(0) = 0.  
\end{align}
In particular, when
$\nu=\zeta(\lambda)$ we have 
$\lambda_1(\nu) =\lambda$, since (by~\eqref{eq:9b} with $z =
\zeta(\lambda)$) $f_{\zeta(\lambda),\lambda}$ will be a positive
eigenfunction of ${\mathfrak k}_{\lambda}(\zeta(\lambda))$.
\end{remark}

\subsection{Main results} 
We are not able to prove~\eqref{eq:27} for all $\lambda\in]\Theta_0,1]$. Here we
state some partial results. Clearly, $\nu=\zeta$ is a stationary point for 
$\lambda_1(\nu)$. Our first result shows that this is a local minimum.   

\begin{theorem}\label{thm:nuzeta}~
\begin{enumerate}
\item Let $\Theta_0<\lambda\leq 1$. Then $\lambda_1(\nu)$ has a local minimum 
for $\nu=\zeta,$ i.e., there exist positive constants 
$\delta_\lambda$ and $c_\lambda$ such that for all 
$|\nu-\zeta|<\delta_\lambda$ it holds that
\[
\lambda_1(\nu)\geq \lambda + c_\lambda(\nu-\zeta)^2.
\]
\item Let $\lambda > \Theta_0,$ $z \in {\mathbb R},$ and let $f_{z,\lambda}$ be 
a positive minimizer of ${\mathcal F}_{z,\lambda}$. 
  Define
  \begin{align}\label{eq:61b}
  \lambda_1(\nu;z) :=  \inf \Spec\Big\{-\frac{d^2}{dt^2} + (t-\nu)^2 + \lambda
      f_{z,\lambda}^2 \Big\} ,
\end{align}
where we consider the Neumann
realization on $L^2({\mathbb R}^+)$ of the operator.

Then, $\lambda_1(\nu;z) \rightarrow 1$ as $\nu \rightarrow +\infty$. 
Furthermore, there exists $\nu_0=  \nu_0(\lambda,z)>0$ such that 
  \begin{equation}
    \label{eq:61}
     \lambda_1(\nu;z)> 1,
  \end{equation}
for all $\nu \geq \nu_0$.
\end{enumerate}
\end{theorem}

\begin{remark}
In particular, the second item in Theorem~\ref{thm:nuzeta} implies
that~\eqref{eq:27} is not true for $\lambda >1$. It is therefore
natural to expect that~\eqref{eq:27} will be valid if and only if
$\lambda \in \left]\Theta_0, 1\right]$. Notice that we will not prove that a
minimum $\zeta(\lambda)$ exists for $\lambda >1$. This explains the
somewhat cumbersome statement in the second item in
Theorem~\ref{thm:nuzeta}. 
\end{remark}

We also obtain an explicit range of values of $\lambda$ for which the 
condition~\eqref{eq:27} is satisfied. The results contain some explicit 
universal constants that will be defined later. In this introduction we will 
only state the numerical values obtained.

\begin{theorem}\label{thm:largenu}~
\begin{itemize}
\item[(i)] Let $\Theta_0<\lambda\leq 1$. For all $\nu\leq 1.33$ it holds that 
$\lambda_1(\nu)\geq\lambda$. 
\item[(ii)] Let $\Theta_0\leq\lambda\leq 0.8$. 
Then~\eqref{eq:27} holds, i.e.
\begin{equation*}
\inf_{\nu\in\mathbb{R}}\lambda_1(\nu)\geq \lambda.
\end{equation*}
\end{itemize}
\end{theorem}

\begin{figure}[H]
\includegraphics{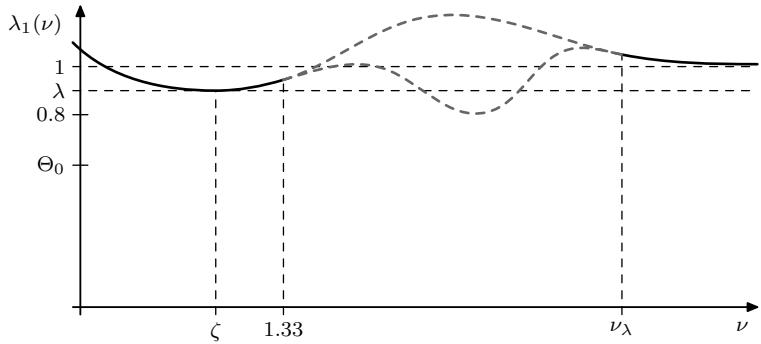}
\caption{A schematic picture of what we know about $\lambda_1(\nu)$ from 
Theorems~\ref{thm:nuzeta} and~\ref{thm:largenu}. The gray dashed parts show two 
possible scenarios.}
\label{fig:whatweknow}
\end{figure}

In Section~\ref{sec:linear} we recall some well-known results about the linear
de Gennes operator, and give some new spectral estimates. In 
Section~\ref{sec:nonlinear} we study the nonlinear problem appearing from the
functional $\mathcal{F}_{z,\lambda}(\phi)$ in~\eqref{eq:44} and 
prove~\eqref{eq:61}. In Section~\ref{sec:mainop} we consider the operator 
$\mathfrak{k}_\lambda(\nu)$ and prove the remainder of Theorem~\ref{thm:nuzeta} 
and Theorem~\ref{thm:largenu}.

\section{The linear problem}\label{sec:linear}

\subsection{Reminder for the de Gennes operator}
Define 
\begin{align} \label{defmathfrakh}
{\mathfrak h}(\xi) = -\frac{d^2}{dt^2} + (t - \xi)^2,
\end{align}
in $L^2({\mathbb R}^{+})$ with Neumann boundary conditions at $0$. We will
denote the eigenvalues of this operator by $\{\mu_j(\xi)\}_{j=1}^{\infty}$ and
corresponding (real normalized) eigenfunctions by $u_j(t) =
u_j(t;\xi)$.

From a similar calculation as the one leading to~(A.18) in~\cite{bohe},
\begin{equation}\label{eq:muonebohe}
\mu_1(\xi) \geq 1 - C_1\xi\exp(-\xi^2),
\end{equation}
for some constant $C_1>0$ and for sufficiently large $\xi$. As part of the 
proof of Proposition~\ref{prop:Asympu} below we will obtain a weaker 
asymptotics of $\mu_1(\xi)$.

A basic identity from perturbation theory (Feynman-Hellmann) is
\begin{align}
  \label{eq:1}
  \mu_j'(\xi) = -2 \int_0^{+\infty}  (t -\xi)  |u_j(t;\xi)|^2\,dt.
\end{align}

An integration by parts, combined with the equation satisfied by
$u_j(t;\xi)$ yields the useful alternative formula from 
Dauge-Helffer~\cite{DaugeHelffer}:
\begin{align}
  \label{eq:4}
  \mu_j'(\xi) = (\xi^2-\mu_j(\xi))  |u_j(0;\xi)|^2\,.
\end{align}
From~\eqref{eq:4} it is simple to deduce that $\mu_j$ has a unique minimum 
attained at $\xi_0^{(j)}$ satisfying 
\begin{align}
  \label{eq:5}
  \mu_j(\xi_0^{(j)}) = (\xi_0^{(j)})^2\,.
\end{align}
Notice that, from~\eqref{eq:1}, we obtain 
\begin{align}
  \label{eq:6}
  \xi_0^{(j)} > 0\,,
\end{align}
for all $j$. We will sometimes write $\xi_0=\xi_0^{(1)}$. By definition
\begin{align}
\label{eqTheta_0}
\Theta_0 = \inf_{\xi\in\mathbb{R}} \mu_1(\xi) = \mu_1(\xi_0^{(1)}) 
= (\xi_0^{(1)})^2.
\end{align}
Finally, we recall that 
\begin{equation}
\mu_j(0) = 1+ 4(j-1) \,,\quad \lambda_j^D(0) = 3 + 4(j-1)\,,
\end{equation}
where $\lambda_j^D(\xi)$ denotes the $j$-th eigenvalue of the
Dirichlet realization of $\mathfrak h(\xi)$ in $L^2(\mathbb R^+)$.
These identities follow upon noticing that the eigenfunctions of the harmonic 
oscillator on the entire line are respectively even or odd functions.

\subsection{Comparison Dirichlet-Neumann}
In this section we recall useful links between the Dirichlet spectrum
and the Neumann spectrum of the family $\mathfrak h(\xi)$ ($\xi \in
\mathbb R$) in $L^2(\mathbb R^+)$\,. By domain monotonicity, it is standard
that $\xi \mapsto \lambda_j^D(\xi)$ is monotonically
decreasing. By comparison of the form domains:
\begin{equation}
\mu_j(\xi) \leq \lambda_j^D(\xi)\,.
\end{equation}
Also,
\begin{gather*}
\lim_{\xi \to +\infty} \lambda_1^D(\xi)= \lim_{\xi \to +\infty}
\mu_1(\xi) = 1\,,\\
\lim_{\xi \to +\infty} \lambda_2^D(\xi)= \lim_{\xi \to +\infty}
\mu_2(\xi) = 3\,.
\end{gather*}
Using Sturm-Liouville theory, we also observe that, for
  any $j\geq 2$ and any  $\xi$, there exists $\xi'$ such that
\begin{equation}
\mu_j(\xi)= \lambda_{j-1}^D(\xi')\,.
\end{equation}
In particular,  using that
\begin{equation}
\inf _{\xi\in\mathbb{R}} \lambda_1^D(\xi)=1\,,
\end{equation}
we get 
\begin{equation}\label{minormu2}
\mu_2(\xi) > 1\,.
\end{equation}

\subsection{The virial theorem}
For $\ell >0$, the map $t \mapsto \ell t$ can be unitarily implemented on 
$L^2(\mathbb{R}^+)$ by the operator $U f(t) = \sqrt{\ell} f(\ell t)$. Therefore, 
$\mathfrak h(\xi)$ is isospectral to the (Neumann realization of the) operator
\begin{equation*}
  {\mathfrak k}_{\ell} := -\ell^{-2} \frac{d^2}{dt^2} + (\ell t- \xi)^2.
\end{equation*}
Since the eigenvalues are unchanged when $\ell$ varies we can take the
derivative at $\ell =1$ and find (using~\eqref{eq:1})
\begin{align*}
  0 &= \int_0^{+\infty} |u_j'(t;\xi)|^2\,dt 
     - \int_0^{+\infty} t (t - \xi) |u_j(t;\xi)|^2\,dt \\
    &= \int_0^{+\infty} |u_j'(t;\xi)|^2\,dt 
     - \int_0^{+\infty} (t - \xi)^2 |u_j(t;\xi)|^2\,dt 
     + \frac{\xi}{2} \mu_j'(\xi).
\end{align*}
Combined with the definition of the energy
\begin{align*}
  \mu_j(\xi) = \int_0^{+\infty} |u_j'(t;\xi)|^2\,dt
             + \int_0^{+\infty} (t- \xi)^2 |u_j(t;\xi)|^2\,dt\,,
\end{align*}
we get
\begin{align}
  \label{eq:2a}
  \int_0^{+\infty} |u_j'(t;\xi)|^2\,dt &= \frac{\mu_j(\xi)}{2}  - \frac{\xi
    \mu_j'(\xi)}{4}\,,
\end{align}
and
\begin{align}\label{eq:2b}
\int_0^{+\infty} (t- \xi)^2 |u_j(t;\xi)|^2\,dt 
&= \frac{\mu_j(\xi)}{2} + \frac{\xi \mu_j'(\xi)}{4}\,.
\end{align}

\subsection{Lower bounds on $\mu_j(\xi)$}

\subsubsection{Estimates on $\mu_1$}
As a warm-up, we recall the lower bound on $\mu_1 (\xi)$. Let $u_1(\,\cdot\,;\xi)$ 
be the ground state of ${\mathfrak h}(\xi)$. We use this function as a trial
state for ${\mathfrak h}(0)$ and find
\begin{align}
1 = \inf {\rm Spec}\, {\mathfrak h}(0) &< \langle u_1(\,\cdot\,;\xi), 
{\mathfrak h}(0) u_1(\,\cdot\,;\xi)\rangle 
=\mu_1(\xi) + 2 \xi \int_0^{+\infty} (t-\xi) u_1(t;\xi)^2 \,dt + \xi^2. \nonumber 
\end{align}
So we obtain the inequality~:
\begin{equation} \label{eq:13}
1 <  \mu_1(\xi) - \xi \mu_1'(\xi) + \xi^2\,. 
\end{equation}
We insert $\xi_0^{(1)}$, using $(\xi_0^{(1)})^2 = \Theta_0 =
\min_{\xi} \mu_1(\xi)$, $\mu_1'(\xi_0^{(1)})=0$ and get
\begin{align}
  \label{eq:14}
  \frac{1}{2} < \Theta_0\,.
\end{align}

\subsubsection{Estimates on $\mu_j$, $j>1$}
From~\eqref{eq:5},~\eqref{eq:6} and the fact that 
$\lim_{\xi\to+\infty}\mu_j(\xi)=(2j-1)$ we find that 
\[
0<\xi^{(j)}_0<\sqrt{2j-1}.
\]
The function $\xi\mapsto \mu_j(\xi)$ decreases from its value $\mu_j(0)=4j-3$  
until it arrives at its minimum at $\xi^{(j)}_0$, after which it becomes 
increasing, so there exists a unique point $\widehat{\xi}_j > 0$ 
such that $\mu_j(\widehat{\xi}_j)=2j-1$. By comparison with the harmonic 
oscillator on a half axis it can be seen that $\widehat{\xi}_j$ coincides with
the smallest value of $\xi$ for which $h_j'(\xi)=0$, where $h_j'(\xi)$ denotes 
the $j$th Hermite function. In particular one easily finds that
\begin{equation}\label{eq:xihat}
\widehat{\xi}_2 = 1,\quad \text{and}\quad 
\widehat{\xi}_3 = \sqrt{5/2}.
\end{equation}
To get the behavior of $\widehat{\xi}_j$ as $j\to\infty$ we observe by 
reflection that $-\widehat{\xi}_j$ is given by the value of $\xi$ 
for which $\mu_1(\xi)=2j-1$. 

Let us get an upper bound on $\mu_1(\xi)$ for $\xi$ negative.
For any $\gamma>0$ and any $\xi\in\mathbb{R}$ we use the inequality
\[
(t-\xi)^2 \leq (1+\gamma)t^2 
+ (1+1/\gamma)\xi^2
\]
to obtain the quadratic form comparison (here and below 
$\int_0^{+\infty} |u|^2\,dt =1$)
\[
\int_0^{+\infty} |u'|^2 + (t-\xi)^2|u|^2\,dt \leq 
\int_0^{+\infty} |u'|^2 + (1+\gamma)t^2|u|^2\,dt 
+ (1+1/\gamma)\xi^2.
\]
Comparing the first eigenvalue $\mu(\xi)$ with the first eigenvalue of the 
(scaled) harmonic oscillator, we find
\[
\mu_1(\xi)\leq \sqrt{1+\gamma}+(1+1/\gamma)\xi^2.
\]
The upper bound we get from this seems to be poor.

For any $\gamma>0$ and any $\xi\in\mathbb{R}$ we use the inequality
\[
(t-\widehat{\xi}_j)^2 \leq (1+\gamma)(t-\xi)^2 
+ (1+1/\gamma)(\widehat{\xi}_j-\xi)^2
\]
to obtain the quadratic form comparison
\[
\int_0^{+\infty} |u'|^2 + (t-\widehat{\xi}_j)^2|u|^2\,dt \leq 
\int_0^{+\infty} |u'|^2 + (1+\gamma)(t-\xi)^2|u|^2\,dt 
+ (1+1/\gamma)(\widehat{\xi}_j-\xi)^2.
\]
By scaling and change of function, we have that the quadratic form on the 
right-hand side is unitary equivalent to
\[
\sqrt{1+\gamma}\int_0^{+\infty} |u'|^2 + (t-(1+\gamma)^{1/4}\xi)^2|u|^2\,dt
+(1+1/\gamma)(\widehat{\xi}_j-\xi)^2.
\]
In particular, with the choice $\xi=\xi^{(j)}_0(1+\gamma)^{-1/4}$ we obtain, 
comparing the $j$th eigenvalue of the corresponding operators and 
using~\eqref{eq:5}, that
\begin{align*}
2j-1 =\mu_j(\widehat{\xi}_j) &\leq \sqrt{1+\gamma}\mu_j(\xi^{(j)}_0)+
(1+1/\gamma)\bigl(\widehat{\xi}_j-\xi^{(j)}_0(1+\gamma)^{-1/4}\bigr)^2\\
& = \sqrt{1+\gamma}\bigl(\xi^{(j)}_0\bigr)^2+
(1+1/\gamma)\bigl(\widehat{\xi}_j-\xi^{(j)}_0(1+\gamma)^{-1/4}\bigr)^2.
\end{align*}
Now let $j=2$. By~\eqref{eq:xihat} we have
\[
3 \leq \sqrt{1+\gamma}\bigl(\xi^{(2)}_0\bigr)^2+
(1+1/\gamma)\bigl(\xi^{(2)}_0(1+\gamma)^{-1/4}-1\bigr)^2.
\]
Completing the square, we get
\[
\bigl(\xi^{(2)}_0-(1+\gamma)^{-3/4}\bigr)^2 
\geq \frac{2\gamma}{(1+\gamma)^{3/2}},
\]
and hence the inequality
\begin{equation}\label{eq:xitwogamma}
\xi^{(2)}_0 > \frac{1+\sqrt{2\gamma}}{(1+\gamma)^{3/4}}
\end{equation}
(since $\frac{1-\sqrt{2\gamma}}{(1+\gamma)^{3/4}} < 1$ for all $\gamma>0$. 
Indeed, the function 
$\gamma\mapsto \frac{1-\sqrt{2\gamma}}{(1+\gamma)^{3/4}}$ starts at $1$ 
for $\gamma=0$ and then decreases to its minimal value $-1/\sqrt{3}$ for 
$\gamma=8$ after which it increases to $0$ as $\gamma\to\infty$).
Optimizing~\eqref{eq:xitwogamma} in $\gamma>0$ we find that the maximal value 
is attained for $\gamma=1/2$, for which we have
\[
\xi^{(2)}_0 > \frac{2^{7/4}}{3^{3/4}}\approx 1.48.
\]
The corresponding lower bound for $\mu_2$ is
\begin{equation}\label{eq:mutwoopt}
\mu_2\bigl(\xi^{(2)}_0\bigr) \geq \frac{2^{7/2}}{3^{3/2}}\approx 2.18.
\end{equation}

Continuing with $j=3$, we arrive at the inequality
\[
5 \leq \sqrt{1+\gamma}\bigl(\xi^{(3)}_0\bigr)^2+
(1+1/\gamma)\bigl(\xi^{(3)}_0(1+\gamma)^{-1/4}-\sqrt{5/2}\bigr)^2.
\]
The same type of calculation shows that
\[
\xi^{(3)}_0 > \sqrt{\frac{5}{2}}\frac{1+\sqrt{\gamma}}{(1+\gamma)^{3/4}}.
\]
Optimizing over $\gamma>0$ yields 
$\gamma=\frac{1}{2}\bigl(13-3\sqrt{17}\bigr)\approx 0.32$ with corresponding 
inequality
\[
\xi^{(3)}_0 
>\frac{\sqrt{5}\Bigl(2+\sqrt{26-6\sqrt{17}}\Bigr)}{(30-6\sqrt{17})^{3/4}}
\approx 2.01
\]
which in turn gives
\[
\mu_3(\xi^{(3)}_0)
\geq \frac{5\Bigl(2+\sqrt{26-6\sqrt{17}}\Bigr)^2}{(30-6\sqrt{17})^{3/2}}
\approx 4.04.
\]

\begin{remark}
We can compare these estimates with the numerical values
\[
\xi^{(2)}_0\approx 1.62,\quad \mu_2(\xi^{(2)}_0)\approx 2.64,\quad
\xi^{(3)}_0\approx 2.16,\quad \text{and}\quad \mu_3(\xi^{(3)}_0)\approx 4.65.
\]
\end{remark}

\subsection{Asymptotics of $u_1$}
We end this section by giving an asymptotic formula for
$u_1(\,\cdot\,;\xi)$ for large $\xi$.

\begin{proposition}\label{prop:Asympu}
For all $\alpha<1$ there exist $C_{\alpha}>0$ and
$\Xi_0>0$ such that
\begin{align}
  \label{eq:62}
  \Bigl|u_1(t,\xi) - \frac{1}{\sqrt{\pi}} \exp\bigl[
  -(t-\xi)^2/2\bigr]\Bigr| \leq C_{\alpha} \exp(-\alpha \xi^2/2),
\end{align}
for all $t \in {\mathbb R}^+$ and all $\xi>\Xi_0$.
\end{proposition}

\begin{proof}
Let $\phi$ be smooth, $\phi(t) = 0$ for $t\leq 1$, $\phi(t) = 1$ for
$t\geq 2$ and define
\begin{align}
  \label{eq:63}
  \tilde u(t) = \phi(t) \frac{1}{\sqrt{\pi}} \exp\big[
  -(t-\xi)^2/2\big].
\end{align}
An elementary calculation now yields (for $\xi>2$ and some constant $C>0$)
\begin{align}
  \label{eq:64}
  \bigl\| [{\mathfrak h}(\xi) - 1 ]\tilde u \bigr\|^2 
\leq C \xi^2 \exp\bigl(-(\xi-2)^2\bigr),
\end{align}
Using the lower bound on $\mu_2(\xi)$ and the spectral theorem this
implies that
\begin{align}
  \label{eq:70}
  |\mu_1(\xi) -1| \leq C \exp(-\alpha \xi^2/2),
\end{align}
and the existence of a (possibly non-normalized) ground state
eigenfunction $u_1$ such that
\begin{align}
  \label{eq:71}
  \| \tilde u -  u_1 \|_2 \leq C \exp(-\alpha \xi^2/2).
\end{align}
One now obtains the similar estimate in $W^{1,2}(\mathbb{R}^+)$, from which the
pointwise estimate follows.
\end{proof}

\section{Estimates on the non-linear problem}\label{sec:nonlinear}
We now analyse the functional ${\mathcal F}_{z,\lambda}$ defined 
in~\eqref{eq:44}.
\subsection{Preliminaries}

We introduce the notation
\begin{align}\label{eq:2}
  {\mathfrak I}(\lambda):= \{ \xi \in {\mathbb R} \,:\, \mu_1(\xi) < \lambda \}.
\end{align}
For future reference, we notice that if $\Theta_0 < \lambda <1$, then
there exist $\xi_1(\lambda), \xi_2(\lambda)>0$ such that
\begin{align}\label{eq:7}
  {\mathfrak I}(\lambda) = \bigl]\xi_1(\lambda), \xi_2(\lambda)\bigr[.
\end{align}
For $\lambda =1$ we have $ {\mathfrak I}(\lambda)=\left[0,\infty\right[$.

\begin{theorem}\label{thm:Sammenkog}~
\begin{itemize}
\item For all $z \in {\mathbb R}, \lambda >0,$ the functional 
${\mathcal F}_{z,\lambda}$ admits a non-negative minimizer 
$f_{z,\lambda} \in {\mathcal B}^1({\mathbb R}^{+})$, which is non-trivial if 
and only if $\lambda > \mu_1(z)$. The minimizer $f_{z,\lambda}$ is a solution 
to the Euler-Lagrange equation~\eqref{eq:9b} and satisfies the bound
\begin{equation}\label{eq:fone}
\|f_{z,\lambda}\|_\infty \leq 1\,.
\end{equation}
Furthermore, minimizers are unique up to multiplication by a constant 
$c \in {\mathbb S}^1 \subset {\mathbb C}$.
\item For all $\epsilon\in \left]0,1/2\right[$, $\lambda >0$ and $z\in 
{\mathfrak I}(\lambda)$, there exist constants $c_{\epsilon}, C_{\epsilon}>0$
  such that
  \begin{align}
    \label{eq:46}
    c_{\epsilon} \exp\Big( -\Big[\frac{1}{2}+\epsilon\Big](t-z)^2\Big) 
\leq f_{z,\lambda}(t) 
\leq C_{\epsilon} \exp\Big( -\Big[\frac{1}{2}-\epsilon\Big](t-z)^2\Big).
  \end{align}
\end{itemize}
\end{theorem}

\begin{proof}
The first item in
Theorem~\ref{thm:Sammenkog} is a slight improvement of known results 
(see~\cite[Proposition~14.2.1 and~14.2.2]{fohebook}), so we will only give 
brief indications of proof. For given $z$ and $\lambda$ the functional is 
clearly bounded from below, so the existence of minimizers is standard. Also, 
by differentiation of the absolute value, we see that minimizers can be chosen 
non-negative. The proof of the non-triviality statement is also 
straight-forward. The equation~\eqref{eq:9b} follows by variation around a 
minimum, and~\eqref{eq:fone} is a consequence of the maximum principle applied 
to~\eqref{eq:9b}.

We finally consider the uniqueness question.
Let $u$ be a minimizer and let $f = |u|$. By the Euler-Lagrange 
equation~\eqref{eq:9b} we see that 
\begin{align}
{\mathfrak k}_{\lambda}(z ) f = \lambda f,\qquad {\mathfrak k}_{\lambda}(z ) u 
= \lambda u.
\end{align}
By Cauchy uniqueness, we therefore have $u = c f$ for some $c \in
{\mathbb S}^1$. Therefore, to prove uniqueness it suffices to prove
uniqueness of non-negative minimizers. The proof of this (which does
not use any bound on the value of $\lambda$) is given in the proof of
\cite[Proposition~14.2.2]{fohebook} and will not be repeated.

The upper and lower bounds in~\eqref{eq:46} can both be proved using the 
following strategy, so we only consider the upper bound.
We start from the equation for $f_{z,\lambda}$ in the form
\begin{equation}
  \label{eq:56}
f_{z,\lambda}''(t) 
= [ (t-z)^2 + \lambda f_{z,\lambda}^2(t) - \lambda] f_{z,\lambda}(t).
\end{equation}
Define, for $\alpha <1$, the function $g$ as $g(t) = C
\exp(-\frac{\alpha}{2}(t-z)^2)$, for some constant $C>0$. Then
\begin{equation}
  \label{eq:57}
  g''(t)= [\alpha^2 (t-z)^2 - \alpha] g(t).
\end{equation}
Choose $T>z$ so large that 
\begin{equation}
  \label{eq:58}
  0<[\alpha^2 (t-z)^2 - \alpha] \leq [ (t-z)^2 + \lambda
  f_{z,\lambda}^2(t) - \lambda] ,
\end{equation}
for all $t \geq T$. This is possible since $\alpha <1$. Choose $C>0$ in such a 
way that 
\begin{align}
  \label{eq:59}
  g(T) > f_{z,\lambda}(T).
\end{align}
Suppose that the inequality $g(t) \geq f_{z,\lambda}(t)$ fails for
some $t > T$. Since both functions tend to $0$ at $+\infty$ (at least
along some sequence, since $f \in L^2({\mathbb R}^+)$), we deduce
that $u:= f- g$ has a positive maximum at some point $t_0>T$. Thus
$u''(t_0) \leq 0$. But, for $t \geq T$, we have
\begin{align}
  \label{eq:60}
  u''(t) &= [ (t-z)^2 + \lambda
  f_{z,\lambda}^2(t) - \lambda] f_{z,\lambda}(t)- [\alpha^2 (t-z)^2 - \alpha] 
g(t) \nonumber \\
&\geq
[\alpha^2 (t-z)^2 - \alpha] u(t).
\end{align}
At $t_0$ this is strictly positive and we get a contradiction.
\end{proof}

By a continuity argument, we find

\begin{proposition}
For $0 < \lambda \leq 1$, the function 
\begin{align}
  \label{eq:45}
  {\mathbb R} \ni z \mapsto {\mathcal F}_{z,\lambda}(f_{z,\lambda})
\end{align}
admits a minimum $\zeta(\lambda) >0$.
\end{proposition}

Notice that for $\lambda > 1$, the existence of a minimum is an open problem.
\begin{proof}
Only the case $\lambda =1$ needs some consideration. We will prove
that the minimal energy in that case tends to $0$ as $z \to
+\infty$. By continuity this implies the proposition. We calculate, for 
arbitrary $\phi \in {\mathcal B}^1({\mathbb R}^+)$ and 
$\alpha \in \left]0,1\right[$, and estimating (part of) the quadratic 
expression from below by the linear ground state energy
\begin{align}
  \label{eq:47}
  {\mathcal F}_{z,1}(\phi) &\geq \int_0^{+\infty} \bigl[\alpha (t -z)^2
  +(1-\alpha)\mu_1(z) - 1\bigr] |\phi|^2 + \frac{1}{2} |\phi|^4
  \,dt\nonumber \\
&\geq \int_{\bigl\{|t-z|\leq \sqrt{[1 -
    (1-\alpha)\mu_1(z)]/\alpha} \bigr\}} \bigl[(1-\alpha)\mu_1(z) - 1\bigr]
|\phi|^2 + \frac{1}{2} |\phi|^4\,dt\nonumber \\
&\geq
- \bigl[(1-\alpha)\mu_1(z) - 1\bigr]^2 \sqrt{\frac{1
  -(1-\alpha)\mu_1(z)}{\alpha}}\nonumber \\
&= - \bigl[1 -\mu_1(z)+\alpha\mu_1(z)\bigr]^2
\sqrt{\frac{1-\mu_1(z) + \alpha\mu_1(z)}{\alpha}
 }
,
\end{align}
where the last inequality follows by completing the square.
We choose $\alpha = \alpha(z) = 1-\mu_1(z) \rightarrow 0$ as $z
\rightarrow +\infty$ to get the conclusion.
\end{proof}

We can now prove~\eqref{eq:61}.
\begin{proof}[Proof of the second item in Theorem~\ref{thm:nuzeta}]
Let $z\in {\mathbb R}$ and let $f_{z,\lambda}$ be a
positive minimizer of ${\mathcal F}_{z,\lambda}$. Notice that $z$ and
$\lambda$ will be fixed in the remainder of the proof. We therefore
write $f$ instead of $f_{z,\lambda}$. We also denote by
$\tilde{\lambda}_j(\nu)=\lambda_j(\nu,z)$ the eigenvalues of the operator 
in~\eqref{eq:61b}.

We apply Temple's inequality (see~\cite{kato3}) with $u_1 := u_1(\,\cdot\,;\nu)$ 
as a test function. Under the condition that
$\tilde{\lambda}_2(\nu)>A$, Temple's inequality says that
\begin{equation}\label{eq:Temple1}
\tilde{\lambda}_1(\nu)\geq A - \frac{B}{\tilde{\lambda}_2(\nu)-A},
\end{equation}
where 
\[
A=\Big\langle u_1,\Big\{-\frac{d^2}{dt^2} + (t-\nu)^2 + \lambda f^2\Big\} u_1 
\Big\rangle
 = \mu_1(\nu) + \lambda \| f u_1 \|_2^2
\]
and 
\[
B   = \Big\|\Big\{-\frac{d^2}{dt^2} + (t-\nu)^2 + \lambda f^2\Big\}
u_1\Big\|_2^2  - A^2
= \lambda^2 \| f^2 u_1 \|_2^2 - \lambda^2 \| f u_1\|_2^4.
\]
Using the upper bound in~\eqref{eq:62} and~\eqref{eq:46}, $\| f u_1 \|_2
\rightarrow 0$ as $\nu \rightarrow \infty$. Since
$\tilde{\lambda}_2(\nu) \geq \mu_2(\nu)$ we see that the condition
$\tilde{\lambda}_2(\nu)>A$ is satisfied for large $\nu$'s, and there
\begin{align}
  \label{eq:65}
  \tilde{\lambda}_1(\nu)\geq \mu_1(\nu) + \lambda \| f u_1 \|_2^2 
- C \lambda^2 \| f^2 u_1 \|_2^2,
\end{align}
for some $C>0$ independent of $\nu$.

Using the upper bounds in~\eqref{eq:62} and~\eqref{eq:46}, we get for
all $0<\alpha < 1$, and large $\nu$,
\begin{align}
  \label{eq:66}
  \| f^2 u_1 \|_2^2 &\leq C \exp(-\alpha \nu^2) + C \int_{-\infty}^{+\infty}
  \exp(-2\alpha(t-z)^2)\exp(-\alpha(t-\nu)^2)\,dt \nonumber\\
 &\leq
C \exp(-\alpha \nu^2) + C'\exp(-2\alpha' \nu^2/3),
\end{align}
where $\alpha' < \alpha$ is arbitrary. 

Without striving for optimality, we make the simple estimate
\begin{align}
  \label{eq:67}
  \| f u_1 \|_2^2 \geq \int_{\nu/2 - 1}^{\nu/2+1} f^2 u_1^2\,dt.
\end{align}
In this interval of integration it follows from~\eqref{eq:62} that
$u_1^2 \geq C \exp\bigl(-(\nu/2 +1)^2\bigr)$ and from~\eqref{eq:46} that $f^2
\geq C \exp(-\beta \nu^2/4)$ for any $\beta >1$. Inserting in the
integral yields, for any $\beta'>1$,
\begin{align}
  \label{eq:68}
  \| f u_1 \|_2^2 \geq C \exp( - \beta' \nu^2/2).
\end{align}
Combining~\eqref{eq:65},~\eqref{eq:66},~\eqref{eq:68} and the
asymptotics of $\mu_1$ from~\eqref{eq:muonebohe} gives that
\begin{align}
  \label{eq:69}
  \tilde{\lambda}_1(\nu) > 1,
\end{align}
for large $\nu$, which is~\eqref{eq:61}.

To prove that $\tilde{\lambda}_1(\nu) \rightarrow 1$, we use the variational 
principle with $u_1=u_1(\,\cdot\,;\nu)$ as a test function. 
Notice that by the lower bound just established, we only need to prove an upper 
bound with limit $1$ at infinity. The variational principle gives
\begin{align}
\tilde{\lambda}_1(\nu) \leq \mu_1(\nu) + \lambda \| f u_1 \|_2^2.
\end{align}
Since we have seen above that $\| f u_1 \|_2 \rightarrow 0$ and 
$\mu_1(\nu) \rightarrow 1$ in the large $\nu$ limit, this implies the upper 
bound required. 
\end{proof}

\subsection{A virial-type result}
The function $f_{\zeta,\lambda}$ satisfies the Euler-Lagrange 
equation~\eqref{eq:9b}. Since, $\zeta = \zeta(\lambda)$ is a minimum for the 
non-linear energy, we get
\begin{align}
  \label{eq:36}
  \int_0^{+\infty} (t-\zeta) f_{\zeta,\lambda}^2 \,dt = 0.
\end{align}
In particular it holds that $\zeta(\lambda)>0$.

Moreover, multiplying~\eqref{eq:9b} by $f_{\zeta,\lambda}$ and integrating, 
we obtain
\begin{align}
  \label{eq:10b}
  \| f_{\zeta,\lambda}' \|_2^2 + \|(t- \zeta) f_{\zeta,\lambda} \|_2^2 
   + \lambda \| f_{\zeta,\lambda} \|_4^4 = \lambda \|f_{\zeta,\lambda} \|_2^2\,.
\end{align}

\begin{lemma}
Assume that $\Theta_0\leq\lambda\leq1$ and that $(\zeta,f_{\zeta,\lambda})$ is 
a minimizer of the functional~\eqref{eq:44}. Then
\begin{align}
\|f'_{\zeta(\lambda),\lambda}\|_2^2 
- \|(t - \zeta(\lambda))f_{\zeta(\lambda),\lambda}\|_2^2
 + \frac \lambda 4 \|f_{\zeta(\lambda),\lambda}\|_4^4 
 &=0\,,\label{eq:Hzero}\\
 2 \|f'_{\zeta(\lambda),\lambda}\|_2^2 
 + \frac{5\lambda}{4} \| f_{\zeta(\lambda),\lambda}\|_4^4
 &= \lambda \|f_{\zeta(\lambda),\lambda}\|_2^2\,,\label{eq:Hone}\\
\intertext{ and }
2 \| (t - \zeta(\lambda))f_{\zeta(\lambda),\lambda}\|_2^2
 + \frac {3\lambda}{ 4 }\| f_{\zeta(\lambda),\lambda}\|_4^4
 &= \lambda \|f_{\zeta(\lambda),\lambda}\|_2^2\,.\label{eq:Htwo}
\end{align}
\end{lemma}

\begin{proof}
By a change of variable and of function in the functional 
${\mathcal F}_{z,\lambda}$ we get a rescaled functional
\[
\phi  \mapsto
\int_0^{+\infty} \rho^2 |\phi'(t)|^2 
 + \Bigl(\frac t \rho - \zeta\Bigr)^2 |\phi(t)|^2
 + \frac {\lambda \rho}{ 2 } |\phi(t)|^4 
 - \lambda |\phi(t)|^2 \,dt
\]
with same infimum. Expressing that the infimum is independent of $\rho$, we 
obtain (using~\eqref{eq:36}) at $\rho=1$ and $\zeta =\zeta(\lambda)$, the 
identity~\eqref{eq:Hzero}. Combining with~\eqref{eq:10b} we also 
get~\eqref{eq:Hone} and~\eqref{eq:Htwo}.
\end{proof}

\subsection{Different bounds on $f_{\zeta,\lambda}$}

\begin{proposition}\label{Prop:UnifNormsf}
Assume that $\Theta_0\leq\lambda\leq 1$ and let $(\zeta,f_{\zeta,\lambda})$ be 
a minimum of the function $(z,f) \mapsto {\mathcal F}_{z,\lambda}(f)$ with 
${\mathcal F}$ defined in~\eqref{eq:44}. Then
\begin{equation}\label{eq:fzero}
f_{\zeta,\lambda}(0)^2 = \frac{2}{\lambda}\bigl(\lambda-\zeta^2\bigr).
\end{equation}
Furthermore,
\begin{equation}\label{eq:41}
2(\lambda-\zeta^2) \leq \lambda \| f_{\zeta,\lambda} \|_{\infty}^2  
   \leq \frac{9}{2^{4/3}}\zeta^{2/3}\lambda^{1/3}
\biggl(\frac{1}{2}
-\frac{5(\lambda-\Theta_0)}{12\zeta^{1/2}
\lambda\|u_1(\,\cdot\,;\xi_0)\|_4^2}\biggr)^{1/3}
(\lambda-\mu_1(\zeta))
\end{equation}
and
\begin{equation}\label{eq:42}
\frac{\lambda-\Theta_0}{\|u_1(\,\cdot\,:\xi_0)\|_4^2} 
\leq \lambda \| f_{\zeta,\lambda} \|_4^2 
\leq \frac{3}{2} \zeta^{1/2} (\lambda - \mu_1(\zeta)).
\end{equation}
\end{proposition}

\begin{remark}
A numerical calculation yields the approximate value 
$\|u_1(\,\cdot\,;\xi_0)\|_4^4 \approx 0.584$. One can also get a lower bound to 
$\|u_1(\,\cdot\,;\xi_0)\|_4^4$ using~\eqref{eq:42}: We have
\begin{displaymath}
\|u_1(\,\cdot\,;\xi_0)\|_4^4 \geq \frac{4}{9}\lim_{\lambda \rightarrow \Theta_0} 
\frac{(\lambda-\Theta_0)^2}
{\zeta(\lambda)\bigl(\lambda - \mu_1(\zeta(\lambda))\bigr)^2} 
=\frac{4}{9\xi_0}\approx 0.579.
\end{displaymath}
\end{remark}

\begin{proof}
The lower bound in~\eqref{eq:41} is an easy consequence of~\eqref{eq:fzero}. 
Both are proved in~\cite{pan2}. We reproduce the short proof for the sake of 
completeness. Indeed, define the function
\[
H(t)=f_{\zeta,\lambda}'(t)^2 - (t-\zeta)^2f_{\zeta,\lambda}(t)^2
+\lambda f_{\zeta,\lambda}(t)^2 -\frac{\lambda}{2}f_{\zeta,\lambda}(t)^4.
\]
A calculation, using~\eqref{eq:9b} shows that 
$H'(t)=-2(t-\zeta)f_{\zeta,\lambda}(t)^2$. By exponential decay it also holds 
that $\lim_{t\to\infty}H(t)=0$. Hence, by~\eqref{eq:36} we have that 
$H(0)=-\int_0^\infty H'(t)\,dt=0$. On the other hand we also have 
$H(0)=(\lambda-\zeta^2)f_{\zeta,\lambda}(0)^2
-\frac{\lambda}{2}f_{\zeta,\lambda}(0)^4$. Since
$f_{\zeta,\lambda}(0)\neq 0$, we get the equality in~\eqref{eq:fzero}.

We continue with the lower bound in~\eqref{eq:42}. By definition we have
\begin{align}
  \label{eq:54}
  - \frac{\lambda}{2} \| f_{\zeta,\lambda} \|_4^4 = {\mathcal
    F}_{\zeta,\lambda}[f_{\zeta,\lambda}] = \inf_{z \in {\mathbb R},
    \phi \in B^1} {\mathcal F}_{z,\lambda}[\phi].
\end{align}
We insert the trial state $z= \xi_0$, $\phi = \rho u_1(\,\cdot\,;\xi_0)$,
with $\rho = \sqrt{(\lambda-\Theta_0)/[\lambda \|
  u_1(\,\cdot\,;\xi_0)\|_4^4]}$, in~\eqref{eq:54}. This yields,
\begin{align}
  \label{eq:55}
   - \frac{\lambda}{2} \| f_{\zeta,\lambda} \|_4^4 \leq -
   \frac{\lambda}{2} \frac{(\lambda-\Theta_0)^2}{\lambda^2 \|
  u_1(\,\cdot\,;\xi_0)\|_4^4}.
\end{align}
This finishes the proof of the lower bound in~\eqref{eq:42}.

Finally, we turn to the upper bounds. Using the variational characterization of 
$\mu_1(\zeta),$ equation~\eqref{eq:10b} implies that
\begin{equation}\label{eq:10badd}
\lambda \| f_{\zeta,\lambda} \|_4^4  
 \leq (\lambda-\mu_1(\zeta)) \| f_{\zeta,\lambda}\|_2^2\,.
\end{equation}
We estimate, using~\eqref{eq:36}, and for $\alpha>1$ (recall that $\zeta >0$),
\begin{align}
  \label{eq:37}
  \| f_{\zeta,\lambda}\|_2^2 &\leq \int_0^{\alpha \zeta}
  |f_{\zeta,\lambda}|^2\,dt + \frac{1}{\zeta(\alpha-1)}
  \int_{\alpha\zeta}^{+\infty} (t-\zeta) |f_{\zeta,\lambda}|^2\,dt 
  \nonumber \\
 & = \int_0^{\alpha \zeta} \frac{\alpha \zeta -t}{\zeta(\alpha-1)}
  |f_{\zeta,\lambda}|^2\,dt\nonumber \\
 &\leq \zeta^{1/2} \sqrt{\frac{\alpha^3}{3(\alpha-1)^2}} 
  \|f_{\zeta,\lambda}\|_4^2. 
\end{align}
We choose the optimal $\alpha = 3$ and implement~\eqref{eq:10badd} to get
\begin{equation}\label{eq:38}
   \| f_{\zeta,\lambda}\|_2^2 \leq  \frac{3}{2}\zeta^{1/2} 
   \|f_{\zeta,\lambda}\|_4^2 \leq \frac{3}{2} \zeta^{1/2} 
   \sqrt{\frac{\lambda - \mu_1(\zeta)}{\lambda}}\| f_{\zeta,\lambda}\|_2, 
\end{equation}
i.e.
\begin{align}
  \label{eq:39}
 \| f_{\zeta,\lambda}\|_2 \leq   \frac{3}{2}\, \zeta^{1/2}\, 
   \sqrt{\frac{\lambda - \mu_1(\zeta)}{\lambda}}.
\end{align}
Combining~\eqref{eq:10badd} and~\eqref{eq:39} yields the upper 
bound~\eqref{eq:42}. 

One easily obtains
\begin{align}\label{eq:35}
  f_{\zeta,\lambda}(t)^3 = - \int_t^{+\infty}
  (f_{\zeta,\lambda}^3)'(\tau) \,d\tau 
  \leq 3 \| f_{\zeta,\lambda}\|_4^2 \|f_{\zeta,\lambda}'\|_2\,.
\end{align}
From~\eqref{eq:Hone},~\eqref{eq:55} and~\eqref{eq:38} we have 
\begin{align}
\|f_{\zeta,\lambda}'\|_2^2 
& = \lambda\biggl(\frac{1}{2}\|f_{\zeta,\lambda}\|_2^2 
- \frac{5}{16}\|f_{\zeta,\lambda}\|_4^4\biggr)\\
& \leq \lambda\|f_{\zeta,\lambda}\|_2^2
\biggl(\frac{1}{2}-\frac{5}{12\zeta^{1/2}} \|f_{\zeta,\lambda}\|_4^2\biggr)\\
& \leq \lambda\|f_{\zeta,\lambda}\|_2^2
\biggl(\frac{1}{2}
-\frac{5(\lambda-\Theta_0)}{12\zeta^{1/2}
\lambda\|u_1(\,\cdot\,;\xi_0)\|_4^2}\biggr),
\end{align}
which combined with~\eqref{eq:10badd},~\eqref{eq:39} and~\eqref{eq:35} implies
\begin{multline}
\label{eq:40}
\lambda \| f_{\zeta,\lambda} \|_{\infty}^2 
\leq \lambda \big\{ 3\|f_{\zeta,\lambda}\|_4^2
\|f_{\zeta,\lambda}'\|_2\big\}^{2/3}\\  
\leq \lambda\Biggl\{3\frac{1}{\sqrt{\lambda}}(\lambda-\mu_1(\zeta))^{1/2}
\|f_{\zeta,\lambda}\|_2
\sqrt{\lambda}\|f_{\zeta,\lambda}\|_2\biggl(\frac{1}{2}
-\frac{5(\lambda-\Theta_0)}{12\zeta^{1/2}
\lambda\|u_1(\,\cdot\,;\xi_0)\|_4^2}\biggr)^{1/2}\Biggr\}^{2/3}\\
\leq \lambda \Biggl\{3(\lambda-\mu_1(\zeta))^{1/2}
\frac{9}{4}\zeta\frac{1}{\lambda}(\lambda-\mu_1(\zeta))
\biggl(\frac{1}{2}
-\frac{5(\lambda-\Theta_0)}{12\zeta^{1/2}
\lambda\|u_1(\,\cdot\,;\xi_0)\|_4^2}\biggr)^{1/2}
\Biggr\}^{2/3}\\
\leq \frac{9}{2^{4/3}}\zeta^{2/3}\lambda^{1/3}
\biggl(\frac{1}{2}
-\frac{5(\lambda-\Theta_0)}{12\zeta^{1/2}
\lambda\|u_1(\,\cdot\,;\xi_0)\|_4^2}\biggr)^{1/3}
(\lambda-\mu_1(\zeta)).
\end{multline}
\end{proof}

\subsection{Bounds on $\zeta(\lambda)$}
It follows from Theorem~\ref{thm:Sammenkog} that $\zeta(\lambda) 
\in {\mathfrak I}(\lambda)$.
These bounds on $\zeta$ can be sharpened considerably.

\begin{lemma}\label{lem:intervals}
Let $\Theta_0 < \lambda \leq 1$. It holds that 
\begin{equation}\label{eq:zetabounds}
\sqrt{\lambda/2}\leq \zeta(\lambda)\leq \sqrt{\lambda}.
\end{equation}
\end{lemma}

\begin{proof}
From~\eqref{eq:fzero} we find that $\zeta^2<\lambda$. Moreover, by the 
bound~\eqref{eq:fone}, $\|f_{\zeta,\lambda}\|_\infty\leq 1$, combined 
with the lower bound~\eqref{eq:41}, we easily obtain the lower bound 
$\zeta(\lambda)\geq \sqrt{\lambda/2}$.
\end{proof}

\begin{remark}\label{rem:zetabound}
The lower bound in Lemma~\ref{lem:intervals} can be improved 
using both the lower and upper bounds in~\eqref{eq:41}, see 
Figure~\ref{fig:zetabound}.
\begin{figure}[H]
\includegraphics[width=0.65\textwidth]{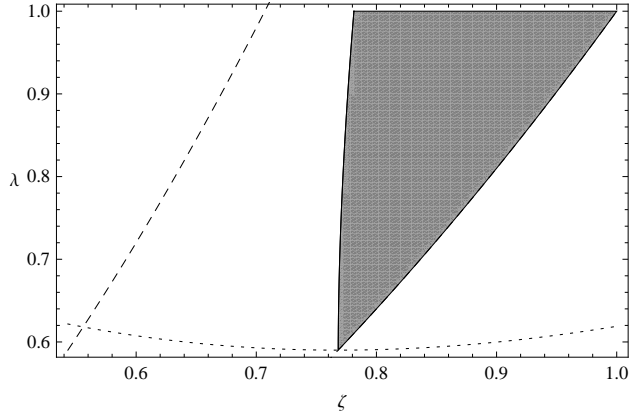}
\caption{Different bounds on $\zeta(\lambda)$. Using Lemma~\ref{lem:intervals}
we find that $\zeta(\lambda)$ should be between the dashed lines. Numerically, 
with the help of~\eqref{eq:41} instead of~\eqref{eq:fone} we find that 
$\zeta(\lambda)$ belongs to the shaded area. The dotted line is the graph of 
$\mu_1(\zeta)$.}
\label{fig:zetabound}
\end{figure}
\end{remark}

\section{The analysis of ${\mathfrak k}_{\lambda}(\nu)$}\label{sec:mainop}

\subsection{Starting point}
Recall the operator ${\mathfrak k}_{\lambda}(\nu)$ with associated eigenvalues 
$\{ \lambda_{j}(\nu)\}$ defined in~\eqref{eq:1b}. We will for shortness write 
$f$ instead of $f_{\zeta(\lambda),\lambda}$ and $\zeta$ instead of 
$\zeta(\lambda)$ in this section.
From the sign of the perturbation and Proposition~\ref{Prop:UnifNormsf} we get:
\begin{proposition}
Let $\Theta_0\leq \lambda\leq 1$. We have the following estimates on the 
eigenvalues of ${\mathfrak k}_{\lambda}(\nu)$:
\begin{equation}\label{eq:45-1}
\mu_j(\nu) \leq \lambda_j(\nu) \leq \mu_j(\nu) 
  + \frac{9}{2^{4/3}}\zeta^{2/3} \lambda^{1/3}
\biggl(\frac{1}{2}
-\frac{5(\lambda-\Theta_0)}{12\zeta^{1/2}
\lambda\|u_1(\,\cdot\,;\xi_0)\|_4^2}\biggr)^{1/3} (\lambda - \mu_1(\zeta)),
\end{equation}
and
\begin{equation}\label{eq:46-1}
\mu_1(\nu) \leq \lambda_1(\nu) \leq \mu_1(\nu) 
               + \frac{3^{3/4}}{2^{1/2}} \zeta^{1/2}
  (\lambda - \mu_1(\zeta)) \bigl(\mu_1(\nu)/2-\nu\mu_1'(\nu)/4\bigr)^{1/4}.
\end{equation}
\end{proposition}

\begin{proof}
The estimate~\eqref{eq:45-1} is an immediate consequence of
\eqref{eq:41}. To show the second estimate~\eqref{eq:46-1}, we notice that
\begin{equation*}
\lambda_1(\nu) \leq \langle u_1 , {\mathfrak k}_{\lambda}(\nu) u_1 \rangle 
= \mu_1(\nu) +  \lambda \| f u_1 \|_2^2,
\leq \mu_1(\nu)+ \lambda \| f \|_4^2 \|u_1 \|_4^2,
\end{equation*}
and
\begin{equation}\label{eq:nagy}
  \|u_1 \|_4^2 \leq \frac{2^{1/2}}{3^{1/4}}\| u_1 \|_2^{3/2} \| u_1' \|_2^{1/2} 
\leq \frac{2^{1/2}}{3^{1/4}} \bigl(\mu_1(\nu)/2-\nu\mu_1'(\nu)/4\bigr)^{1/4}.
\end{equation}
The first inequality in~\eqref{eq:nagy} is due to Nagy~\cite{na}, while the 
second one follows from~\eqref{eq:2a}. The upper bound in~\eqref{eq:46-1} 
now follows from the upper bound in~\eqref{eq:42}.
\end{proof}

\begin{lemma}
If $\nu\not\in \mathfrak{I}(\lambda)$ then $\lambda_1(\nu) \geq \lambda$.
\end{lemma}

\begin{proof}
If $\nu\not\in\mathfrak{I}(\lambda)$ then, by~\eqref{eq:45-1}, we 
get $\lambda_1(\nu)\geq \mu_1(\nu)\geq \lambda$.
\end{proof}

We continue with some identities.

\begin{proposition}
Suppose that $\nu_0$ is a stationary point for $\lambda_1$, i.e.
\begin{align}\label{eq:2bis}
  \lambda_1'(\nu_0) = 0\,.
\end{align}
Then we have the following identities:
\begin{gather}
  \{ \lambda_1(\nu_0) - \nu_0^2 -\lambda f^2(0) \} v^2_1(0;\nu_0) 
  = 2\lambda \int_0^{+\infty}
  v^2_1(t;\nu_0) f(t) f'(t)\,dt\,,\label{eq:3b}\\
\int_0^{+\infty} (t- \nu_0) v^2_1(t;\nu_0) \,dt=0\,,\label{eq:3afh}\\
 \| (t-\nu_0)  v_1(\,\cdot\,;\nu_0)\|_2^2 
+ \lambda \int_0^{+\infty} t v^2_1(t;\nu_0)
 f(t) f'(t)\,dt = \| v_1'(\,\cdot\,;\nu_0)\|_2^2\,,  \label{eq:4b}\\
 \| v_1'(\,\cdot\,;\nu_0)\|_2^2 + \| (t-\nu_0) v_1(\,\cdot\,;\nu_0)\|_2^2 
    + \lambda \| f \, v_1(\,\cdot\,;\nu_0) \|_2 ^2=\lambda_1(\nu_0)\,.
\label{eq:5b}
\end{gather}
\end{proposition}

\begin{proof}
Equation~\eqref{eq:3b} is a Dauge-Helffer type formula,~\eqref{eq:3afh} is the
Feynman-Hell\-mann formula,~\eqref{eq:4b} follows by the virial theorem 
and~\eqref{eq:5b} is just the energy equation.
\end{proof}

\begin{corollary}
If $0<\zeta<\nu_0$, $\lambda_1'(\nu_0)=0$ and 
$\int_0^{+\infty} v^2_1(t;\nu_0) f(t) f'(t)\,dt\geq 0$ then 
$\lambda_1(\nu) > \lambda$.
\end{corollary}

\begin{proof}
From~\eqref{eq:3b} and~\eqref{eq:fzero} we get
\begin{equation*}
\lambda_1(\nu_0) \geq \lambda f(0)^2+\nu_0^2 = 
\lambda + (\lambda-\zeta^2)+(\nu_0^2-\zeta^2) > \lambda,
\end{equation*}
since $\lambda\geq \zeta^2$ by~\eqref{eq:zetabounds} and $\nu_0^2>\zeta^2$ by 
the assumption.
\end{proof}

\begin{remark}
From Theorem~\ref{thm:largenu} we notice that it is enough to consider 
$\nu_0>1.33$ and so the condition on $\nu_0$ and $\zeta$ is not restricting 
since $\zeta<1$.

It is also worth to notice that if 
$\int_0^{+\infty} v^2_1(t;\nu_0) f(t) f'(t)\,dt < 0$ then also
\begin{equation*}
\int_0^{+\infty} t v^2_1(t;\nu_0) f(t) f'(t)\,dt < 0,
\end{equation*}
since there exists a $t_0$ such that $f'(t)$ is positive for 
$t\in\left]0,t_0\right[$ and negative for $t\in\left]t_0,\infty\right[$, 
see~\cite{pan2}.
\end{remark}

\subsection{Lower bound on $\lambda_1(\nu)$}

\begin{lemma}
If $\lambda_2(\nu)>\lambda+(\nu-\zeta)^2$ then it holds that
\begin{equation}\label{eq:temple}
\lambda_1(\nu)\geq \lambda + (\nu-\zeta)^2
\biggl[1 - \frac{4\|(t-\zeta)f\|_2^2}
{\bigl(\lambda_2(\nu)-\lambda-(\nu-\zeta)^2\bigr)\|f\|_2^2}\biggr].
\end{equation}
\end{lemma}

\begin{proof}
The Temple inequality (see~\cite{kato3}) with $f/\|f\|_2$ as trial state, 
implies that if $\lambda_2(\nu)>A$ then
\begin{equation}\label{eq:Temple}
\lambda_1(\nu)\geq A - \frac{B}{\lambda_2(\nu)-A},
\end{equation}
where 
\[
A=\frac{\langle f,{\mathfrak k}_{\lambda}(\nu)f \rangle}{\|f\|_2^2}
 = \lambda + (\nu-\zeta)^2
\]
and 
\[
B = \frac{\langle f, ({\mathfrak k}_{\lambda}(\nu)-A)^2 f\rangle}{\|f\|_2^2}
  = \frac{\langle f, {\mathfrak k}_{\lambda}(\nu)^2 f\rangle}{\|f\|_2^2} - A^2.
\]
Using that ${\mathfrak k}_{\lambda}(\zeta)f=\lambda f$, we find that
\[
{\mathfrak k}_{\lambda}(\nu)f = \lambda f - 2(\nu-\zeta)(t-\zeta)f 
+ (\nu-\zeta)^2 f,
\]
and so
\[
\|{\mathfrak k}_{\lambda}(\nu)f\|^2 = 
\bigl(\lambda+(\nu-\zeta)^2\bigr)^2\|f\|_2^2 + 4(\nu-\zeta)^2\|(t-\zeta)f\|_2^2.
\]
We conclude that
\[
B = 4(\nu-\zeta)^2\frac{\|(t-\zeta)f\|_2^2}{\|f\|_2^2}.
\]
Inserting these expressions for $A$ and $B$ into~\eqref{eq:Temple} 
yields~\eqref{eq:temple}.
\end{proof}

\begin{proof}[Proof of Theorem~\ref{thm:nuzeta}]
We only consider (1), since the second item has already been established.
Combining the lower bounds on $\|f\|_4$ from~\eqref{eq:38} and~\eqref{eq:42} we
first get
\begin{equation}
\begin{aligned}
2\|(t-\zeta)f\|_2^2 &= \lambda\|f\|_2^2-\frac{3\lambda}{4}\|f\|_4^4\\
& \leq \lambda\|f\|_2^2\Bigl(1-\frac{1}{2\zeta^{1/2}}\|f\|_4^2\Bigr)\\
& \leq \lambda\|f\|_2^2
\biggl(1-\frac{\lambda-\Theta_0}
{2\lambda\zeta^{1/2}\|u_1(\,\cdot\,;\xi_0)\|_4^2}\biggr).
\end{aligned}
\end{equation}
We implement this in~\eqref{eq:temple} and use the simple inequality 
$\lambda_2(\nu)\geq\mu_2(\nu)$,
\begin{equation}\label{eq:nuzetabracket}
\lambda_1(\nu)\geq \lambda+(\nu-\zeta)^2
\Biggl[
\frac{\mu_2(\nu)-\bigl(3\lambda-
\frac{\lambda-\Theta_0}{\zeta^{1/2}\|u_1(\,\cdot\,;\xi_0)\|_4^2} \bigr)
-(\nu-\zeta)^2}{\lambda_2(\nu)-\lambda-(\nu-\zeta)^2}
\Biggr].
\end{equation}
By continuity it suffices to check verify that
\begin{align}
\label{eq:Numerator}
\mu_2(\zeta) - \biggl(3\lambda-
\frac{\lambda-\Theta_0}{\zeta^{1/2}\|u_1(\,\cdot\,;\xi_0)\|_4^2} \biggr) > 0,
\end{align}
and
\begin{align}
\lambda_2(\zeta)-\lambda>0.
\end{align}
This last inequality is trivially satisfied since $\lambda_2\geq \mu_2$ which 
satisfies the lower bound~\eqref{eq:mutwoopt}. Thus we only have to 
consider~\eqref{eq:Numerator}. Notice that the parenthesis 
in~\eqref{eq:Numerator} is strictly less than $3$. Since $\mu_2$ is decreasing 
on $\left[0,1\right]$ and $\mu_2(1)=3$ this finishes the proof.
\end{proof}

Define the set ${\mathcal X}(\lambda) \subset {\mathfrak I}
(\lambda)$ as the possible values of $\zeta$, i.e.
\begin{align}
{\mathcal X}(\lambda) := \{ \zeta \in {\mathbb R} \,:\, 
\text{ the function } {\mathbb R}\ni z \mapsto 
{\mathcal F}_{z,\lambda}(f_{z,\lambda}) \text{ has a minimum at } \zeta\}.
\end{align}
By Lemma~\ref{lem:intervals} we have ${\mathcal X}(\lambda)  \subset 
\bigl[\sqrt{\lambda/2}, \sqrt{\lambda}\,\bigr]$, but from 
Figure~\ref{fig:zetabound} it actually follows that 
\begin{equation}\label{eq:bestzeta}
{\mathcal X}(\lambda) \subset \bigl[\xi_0, \sqrt{\lambda}\,\bigr]
\end{equation}

We can summarize the result~\eqref{eq:nuzetabracket} of Temple's inequality 
as follows
\begin{proposition}\label{prop:numden}
Let $\Theta_0 \leq \lambda \leq 1$. Assume that
\begin{equation}\label{eq:nume}
\mu_2(\nu)-\biggl(3\lambda-
\frac{\lambda-\Theta_0}{\zeta^{1/2}\|u_1(\,\cdot\,;\xi_0)\|_4^2} \biggr)
-(\nu-\zeta)^2 \geq 0,
\end{equation}
and
\begin{equation}
  \label{eq:72}
\mu_2(\nu)-\lambda-(\nu-\zeta)^2 > 0
\end{equation}
for all $\zeta\in \mathcal{X}(\lambda)$ and $\nu\in\mathfrak{I}(\lambda)$.
Then $ \lambda_1(\nu) \geq \lambda$ for all $\nu \in {\mathfrak I}
(\lambda)$.
\end{proposition}

\begin{proof}[Proof of Theorem~\ref{thm:largenu}]
We will use Proposition~\ref{prop:numden}. We start by verifying~\eqref{eq:72}.
To prove (i) we need only to consider $0\leq\nu\leq 1.33$ and to prove (ii) it
suffices to consider $0\leq\nu\leq 1.5$ since the right endpoint of the 
interval $\mathfrak{I}(0.8)$ is less than $1.5$ (solving the equation 
$\mu_1(\nu)=0.8$ gives a numerical value $\nu\approx 1.496$). 
The inequality~\eqref{eq:72} holds for all $0\leq\nu\leq 1.5$, 
$\zeta\in\mathcal{X}(\lambda)$ and $\Theta_0\leq\lambda\leq 1$. Indeed, 
$(\nu-\zeta)^2<1$ by~\eqref{eq:bestzeta} and $\mu_2(\nu)\geq 2.18$ 
by~\eqref{eq:mutwoopt}.

We now consider~\eqref{eq:nume}. If $0\leq\nu\leq\zeta$, 
$\xi_0\leq \zeta\leq1$ and $\Theta_0\leq\lambda\leq 1$ then
\begin{equation}\label{eq:nulessze}
\mu_2(\nu)-3\lambda-(\nu-\zeta)^2 \geq \mu_2(\nu)-3-(\nu-1)^2.
\end{equation}
From Figure~\ref{fig:mu12prime} it is clear that $\mu_2'(\nu)<2(\nu-1)$ on 
$0\leq\nu\leq1$. Hence, the function $\nu\mapsto \mu_2(\nu)-3-(\nu-1)^2$ is
decreasing on this interval. Since $\mu_2(1)=3$ we find that the right-hand side 
of~\eqref{eq:nulessze} is bounded from below by $0$, and it follows 
that~\eqref{eq:nume} holds for $\nu\leq\zeta$.

To complete the proof of (i) it is sufficient to show that the 
inequality~\eqref{eq:nume} holds for $\xi_0\leq\zeta\leq \sqrt{\lambda}$, 
$\zeta<\nu\leq 1.33$ and $\Theta_0\leq\lambda\leq 1$. From 
Figure~\ref{fig:mu1andmu2} we note that $\mu_2$ is decreasing for these values 
of $\nu$, and so since $\nu>\zeta$ it follows that the left-hand side 
in~\eqref{eq:nume} is decreasing as a function of $\nu$. Hence we get a lower
bound replacing $\nu$ by the right endpoint $1.33$. Moreover, 
$1.5\approx 3-1/\bigl(\xi_0^{1/2}\|u(\,\cdot\,;\xi_0)\|_4^2\bigr)>0$ so we also 
get a lower bound if we replace $\lambda$ by $1$, i.e.
\begin{multline}\label{eq:nuzetai}
\mu_2(\nu)-\biggl(3\lambda-
\frac{\lambda-\Theta_0}{\zeta^{1/2}\|u_1(\,\cdot\,;\xi_0)\|_4^2} \biggr)
-(\nu-\zeta)^2\\
\geq \mu_2(1.33)-
\biggl(3- \frac{1-\Theta_0}{\zeta^{1/2}\|u_1(\,\cdot\,;\xi_0)\|_4^2}\biggr)
-(1.33-\zeta)^2.
\end{multline}
Differentiating the right-hand side of~\eqref{eq:nuzetai} with respect to 
$\zeta$ and estimating on $\xi_0\leq\zeta\leq 1$ we find
\begin{gather*}
\frac{d}{d\zeta}\biggl[
\mu_2(1.33)-
\biggl(3- \frac{1-\Theta_0}{\zeta^{1/2}\|u_1(\,\cdot\,;\xi_0)\|_4^2}\biggr)
-(1.33-\zeta)^2\biggr]\\
\begin{aligned}
&=-\frac{1-\Theta_0}{2\zeta^{3/2}\|u_1(\,\cdot\,;\xi_0)\|_4^2}+2(1.33-\zeta)\\
&\geq 
-\frac{1-\Theta_0}{2\xi_0^{3/2}\|u_1(\,\cdot\,;\xi_0)\|_4^2}+2(1.33-1)\\
&\approx 0.26.
\end{aligned}
\end{gather*}
Thus, we get a lower bound of the right-hand side of~\eqref{eq:nuzetai} by
inserting the left endpoint $\zeta=\xi_0$. The lower bound is 
\[
\mu_2(1.33)-
\biggl(3- \frac{1-\Theta_0}{\xi_0^{1/2}\|u_1(\,\cdot\,;\xi_0)\|_4^2}\biggr)
-(1.33-\xi_0)^2\approx 0.01.
\]
This finishes the proof of (i).

We continue with (ii). It is sufficient to show that the 
inequality~\eqref{eq:nume} holds for $\xi_0\leq\zeta\leq \sqrt{\lambda}$, 
$\zeta<\nu\leq 1.5$ and $\Theta_0\leq \lambda\leq 0.8$, where the endpoint $1.5$ 
is chosen to be slightly larger than the right endpoint of the interval 
$\mathfrak{I}(0.8)$.

Again $\mu_2$ is decreasing for these values 
of $\nu$, and so since $\nu>\zeta$ it follows that the left-hand side 
in~\eqref{eq:nume} is decreasing as a function of $\nu$. Hence we get a lower
bound replacing $\nu$ by $1.5$. In the same way as for (i) we also get a lower 
bound if we replace $\lambda$ by the right endpoint $0.8$, i.e.
\begin{multline}\label{eq:nuzetaii}
\mu_2(\nu)-\biggl(3\lambda-
\frac{\lambda-\Theta_0}{\zeta^{1/2}\|u_1(\,\cdot\,;\xi_0)\|_4^2} \biggr)
-(\nu-\zeta)^2\\
\geq \mu_2(1.5)-
\biggl(3\times 0.8 
- \frac{0.8-\Theta_0}{\zeta^{1/2}\|u_1(\,\cdot\,;\xi_0)\|_4^2}\biggr)
-(1.5-\zeta)^2.
\end{multline}
We differentiate the right-hand side of~\eqref{eq:nuzetaii}, and estimate for
$\xi_0\leq\zeta\leq\sqrt{0.8}$, to find
\begin{gather}
\frac{d}{d\zeta}\biggl[
\mu_2(1.5)-
\biggl(3\times 0.8 
- \frac{0.8-\Theta_0}{\zeta^{1/2}\|u_1(\,\cdot\,;\xi_0)\|_4^2}\biggr)
-(1.5-\zeta)^2
\biggr]\\
\begin{aligned}
&= -\frac{0.8-\Theta_0}{2\zeta^{3/2}\|u_1(\,\cdot\,;\xi_0)\|_4^2}+2(1.5-\zeta)\\
&\geq -\frac{0.8-\Theta_0}{2\xi_0^{3/2}\|u_1(\,\cdot\,;\xi_0)\|_4^2}
+2(1.5-\sqrt{0.8})\\
&\approx 1.0.
\end{aligned}
\end{gather}
Hence, we get a lower bound of the right-hand side of~\eqref{eq:nuzetaii} by
inserting the left endpoint $\zeta=\xi_0$. The lower bound we get is
\[
\mu_2(1.5)-
\biggl(3\times 0.8 
- \frac{0.8-\Theta_0}{\xi_0^{1/2}\|u_1(\,\cdot\,;\xi_0)\|_4^2}\biggr)
-(1.5-\xi_0)^2\approx 0.026.
\]
This finishes the proof of (ii).
\end{proof}

%\begin{figure}[H]
%\includegraphics[width=0.65\textwidth]{zetaplot}
%\caption{Graphs of the right-hand sides of~\eqref{eq:nuzetai} 
%(solid, $\xi_0\leq \zeta\leq 1$) and~\eqref{eq:nuzetaii} 
%(dashed, $\xi_0\leq\zeta\leq \sqrt{0.8}$).}
%\label{fig:zetaplot}
%\end{figure}

\appendix

\section{Comments on the numerical calculations}\label{sec:numerical}

We give some details on how the numerical calculations were done. The solutions 
to the eigenvalue equation $\mathfrak{h}(\xi)u=\mu(\xi) u$, not taking the 
Neumann boundary condition into account, are given by
\begin{equation}\label{eq:usol}
u(t)= c_1 e^{-\frac12(t-\xi)^2}H_{\frac12(\mu(\xi)-1)}(t-\xi) + 
c_2 e^{\frac12(t-\xi)^2}H_{-\frac12(\mu(\xi)+1)}(i(t-\xi)).
\end{equation}
Here, $H_\nu(t)$ solves the Hermite equation (see Section~10.13 in~\cite{erd2})
\begin{equation*}
-y''(t)+2ty'(t)-2\nu y(t)=0,
\end{equation*}
and is polynomially bounded at infinity. Hence, for the 
function $u$ in~\eqref{eq:usol} to be square integrable, we must set $c_2=0$. 
Using the well-known relations for the derivative of $H_\nu$, 
$\frac{d}{dt}H_\nu(t)=2\nu H_{\nu-1}(t)$, we find that the Neumann condition 
$u'(0)=0$ reads
\begin{equation}\label{eq:muxi}
(\mu(\xi)-1)H_{\frac12(\mu(\xi)-3)}(-\xi)+\xi H_{\frac12(\mu(\xi)-1)}(-\xi)=0.
\end{equation}
Hence, for $\xi\in\mathbb{R}$, the $j$th eigenvalue $\mu_j(\xi)$ of the operator
$\mathfrak{h}(\xi)$ is given by the $j$th (positive) solution $\mu(\xi)$ 
of~\eqref{eq:muxi}. To obtain an equation for $\mu_j'(\xi)$ we 
differentiate~\eqref{eq:muxi} implicitly.

We use the software Mathematica from Wolfram Research (who claims that 
Mathematica is able to calculate these special functions to any given 
precision\footnote{See \url{http://reference.wolfram.com/mathematica/ref/HermiteH.html}.})
to solve these equations numerically and draw the plots. By 
inserting~\eqref{eq:5} into~\eqref{eq:muxi} we are also able to calculate the
constant $\Theta_0$ to any precision (see also Remark~A.6 in~\cite{fope}).

\section{Additional graphs}\label{sec:plots}

In this appendix we have collected some additional graphs that have to do with 
the eigenvalues $\mu_j(\xi)$ of $\mathfrak{h}(\xi)$.

\begin{figure}[H]
\centering
\includegraphics[width=0.65\textwidth]{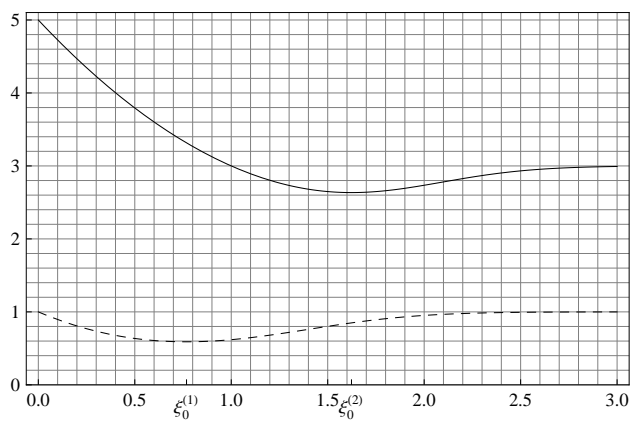}%
\caption{A plot of $\mu_1$ (dashed) and $\mu_2$ (solid).}
\label{fig:mu1andmu2}
\end{figure}

\begin{figure}[H]
\centering
\includegraphics[width=0.65\textwidth]{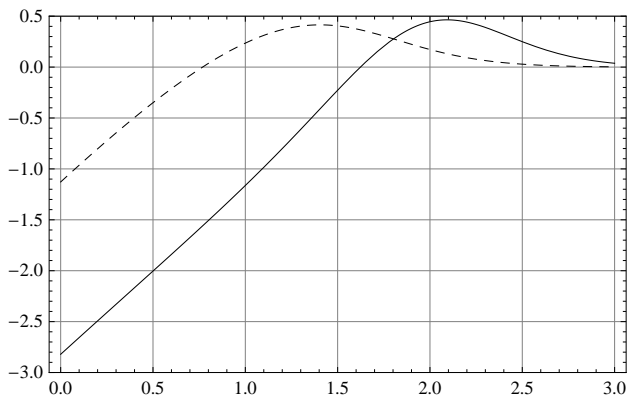}%
\caption{A plot of $\mu_1'$ (dashed) and $\mu_2'$ (solid).}
\label{fig:mu12prime}
\end{figure}

\section*{Acknowledgements}
SF and MP were supported by the Lundbeck Foundation and by the European Research 
Council under the European Community's Seventh Framework Program 
(FP7/2007--2013)/ERC grant agreement 202859.

\nocite{foka,hepe,helf}
%\bibliographystyle{abbrv}
%\bibliography{Spectral}

\end{document}